\documentclass[12pt]{article}

\usepackage{geometry}
\usepackage{enumitem}
\setlist[enumerate,1]{leftmargin=1.5cm}
\setlist[enumerate,2]{label=(\alph*),itemsep=1em}
\usepackage{centernot}
\usepackage{verbatim} 
\usepackage{amssymb, amsfonts, amsmath, amsthm}
\usepackage{mathtools}
\usepackage{tabu,booktabs}
\usepackage{marvosym}
\usepackage{tikz-cd}
\usepackage{verbatim}
\usepackage{amsfonts}
\usepackage{amsmath}
\usepackage{amsthm}
\usepackage{amssymb}
\usepackage{mathrsfs}
\usepackage[fit]{truncate}
\usepackage{multirow}
\usepackage{colortbl}
\usepackage{algpseudocode, algorithmicx}
\usepackage[normalem]{ulem}
\usepackage[bookmarksopen=true]{hyperref}
\usepackage{bookmark}
\bookmarksetup{numbered}

\newcommand{\cE}{\mathcal{E}}
\newcommand{\cU}{\mathcal{U}}
\newcommand{\cV}{\mathcal{V}}

\newcommand{\cC}{\mathcal{C}}

\newcommand{\cO}{\mathcal{O}}

\newcommand{\set}[1]{\left\{#1\right\}}

\renewcommand{\max}[1]{\text{max}\set{#1}}
\renewcommand{\min}[1]{\text{min}\set{#1}}

\newcommand{\rect}{\mathrm{Rect}}

\newcommand{\fold}{\mathit{F}}
\newcommand{\symmdiff}{\; \triangle \;}
\newcommand{\compl}[1]{\overline{#1}}

\newcommand{\SYT}{\mathrm{SYT}}
\newcommand{\fhalfN}{\mathchoice
{\lfloor\tfrac{N}{2}\rfloor}%
{\lfloor\tfrac{N}{2}\rfloor}%
{\lfloor\frac{N}{2}\rfloor}%
{\lfloor\frac{N}{2}\rfloor}}
\newcommand{\chalfN}{\mathchoice
{\lceil\tfrac{N}{2}\rceil}%
{\lceil\tfrac{N}{2}\rceil}%
{\lceil\frac{N}{2}\rceil}%
{\lceil\frac{N}{2}\rceil}}

\newcommand{\vc}[1]{\ensuremath{\vcenter{\hbox{#1}}}}

\theoremstyle{theorem}
\newtheorem{thm}{Theorem}[section]
\newtheorem{prop}[thm]{Proposition}

\newtheorem{lemma}[thm]{Lemma}

\theoremstyle{definition}
\newtheorem{algo}[thm]{Algorithm}
\newtheorem{defn}[thm]{Definition}
\newtheorem{ex}[thm]{Example}
\newtheorem*{ack}{Acknowledgements}

\theoremstyle{remark}
\newtheorem{remark}[thm]{Remark}

\newcommand{\webdist}{\mathrm{webdist}}
\newcommand{\arcdist}{\mathrm{arcdist}}
\newcommand{\arcset}{\mathrm{Arcs}}
\newcommand{\cosep}{\mathrm{CS}}

\newcommand{\domtab}{D}
\newcommand{\rdomtab}{\widetilde{D}}
\newcommand{\symtab}{T}
\newcommand{\compresstab}{C}
\newcommand{\web}{\mathcal{W}}
\newcommand{\webcross}{\mathcal{W}^\times}
\newcommand{\webfromdomtab}[1]{\web_{\textnormal{\texttt{Unfold}}(#1)}}

\newcommand{\mdiag}{\mathcal{M}}
\newcommand{\mdiagcross}{\mathcal{M}^{\times}}

\definecolor{notecolor}{RGB}{255, 5, 226}

\newcommand{\note}[1]{{\color{notecolor}#1}}

\usepackage{graphicx}
\graphicspath{ {figures/} }

\usepackage{young}
\usepackage{xcolor}

\usepackage[utf8]{inputenc}
\usepackage[english]{babel}

\title{Folding rotationally symmetric 
tableaux
via webs}
\author{Kevin Purbhoo\thanks{Research supported by NSERC Discovery Grants 
RGPIN-355462 and RGPIN-04741-2018.}~ and Shelley Wu\thanks{Research supported by NSERC Undergraduate Student Research Award.}}

\begin{document}
\maketitle

\begin{abstract}
Rectangular standard Young tableaux with $2$ or $3$ rows are in bijection 
with $U_q(\mathfrak{sl}_2)$-webs and $U_q(\mathfrak{sl}_3)$-webs respectively.
When $\web$ is a web with a reflection symmetry, the corresponding
tableau $T_\web$ has a rotational symmetry.  Folding $T_\web$
transforms it into a domino tableau $D_\web$. 
We study the relationships between these correspondences.  
For $2$-row tableaux, folding a rotationally symmetric tableau corresponds
to ``literally folding'' the web along its axis of symmetry.
For $3$-row tableaux, we give simple algorithms, which 
provide direct bijective maps 
between symmetrical webs and domino tableaux (in both directions).  
These details of these
algorithms reflect the intuitive idea that 
$D_\web$ corresponds to ``$\web$ modulo symmetry''.
\end{abstract}




\section{Introduction} 

Folding is a bijection on standard Young tableaux, closely related to 
the fundamental operations, promotion and evacuation,
introduced by Sch\"utzenberger \cite{sch}.
It can be thought of as the third operation in this sequence.
Promotion ``cycles'' the entries of a tableau in a particular
way, and often arises in a context where it is performed repeatedly.
Evacuation can be viewed as repeatedly performing promotion, while fixing
the largest entry after each iteration.  Folding is defined similarly
to evacuation, except that we fix the largest \emph{two} entries after each
iteration of promotion.  One could continue the sequence, fixing the
largest three entries after each iteration, and so on, but beyond folding,
these operations do not appear to have any notable properties.
Promotion and evacuation and have been extensively studied 
(see survey article \cite{rstanley}), and folding has
its own unique and peculiar properties that make it an interesting
combinatorial operation.

The importance of the folding operation comes from the 
fact that it restricts to a bijection between certain special classes 
of tableaux.  The best known example is that folding a
self-evacuating tableau yields a domino tableau; this defines a bijection
between these two classes of tableaux \cite{vL}.
Folding also arises in shifted tableau theory.  It gives the bijection 
behind Shor's conjecture (proved by Haiman \cite{haiman-mixed}), 
which expresses 
a Schur $P$-function as a linear combination of ordinary Schur functions.
It also defines a bijection between unshifted and shifted tableaux of 
staircase shape \cite{purbhoo-lg}.  In each of these cases, there is a 
geometric explanation for these bijections,
in which folding can be thought of as simple change of coordinates
which transforms one geometric symmetry into another geometric symmetry
\cite{purbhoo-ribbon, purbhoo-lg}.

Rectangular tableaux with $2$ or $3$ rows
are of special interest because they are in
bijection with planar diagrams called webs.  
Webs were introduced by Kuperberg in \cite{kuperberg} to construct invariant
vectors in representations of the quantum groups $U_q(\mathfrak{sl}_2)$ and
$U_q(\mathfrak{sl}_3)$.
These associated webs are called $U_q(\mathfrak{sl}_2)$-webs and 
$U_q(\mathfrak{sl}_3)$-webs
respectively (or $A_1$-webs and $A_2$-webs after their Lie-types);
here, we will refer to them simply as $2$-webs and $3$-webs.
It is well-known how promotion and evacuation behave under the bijection
with webs: promotion corresponds to cyclically rotating the web \cite{PPR},
and evacuation corresponds to reflecting the web \cite{PaPe}.
The purpose of this paper is to study the properties of folding under 
the bijection between tableaux and webs.

A $2$-web is a non-crossing arc diagram, with vertices labelled
by the entries of the corresponding tableau.  The bijection with rectangular
$2$-row 
tableau is well-known and straightforward: left ends of arcs are labelled by
entries in the first row of the tableau, and right ends are labelled by
entries in the second row.  Here is an example of
a $2 \times n$ standard Young tableau $\symtab$ and the corresponding $2$-web
$\web_\symtab$.
\[ 
T = \begin{young}[c]
1 & 3 & 4& 7 \\ 2 & 5 & 6 & 8 
\end{young} \qquad \qquad 
\vc{\includegraphics[scale=0.225]{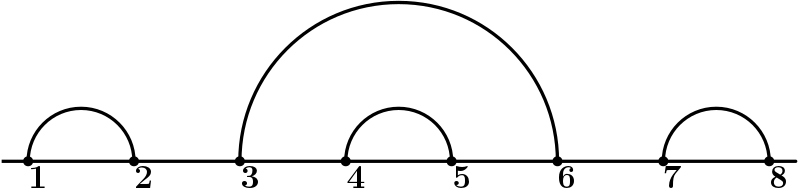}}
\]  
In this case $T$ is self-evacuating.  For rectangular tableaux this
simply means that the tableau is \emph{rotationally symmetric}:
the sum of the entries in $180^\circ$-rotationally opposite positions is 
constant (e.g. for $T$ above, $1$ is opposite $8$, $2$ is opposite $7$, 
etc.).  This implies that the corresponding $2$-web has vertical axis of 
symmetry.

Folding a rotationally symmetric $2 \times n$ 
tableau translates
into a simple-to-describe 
operation on the corresponding symmetrical web.  It can 
be thought of as literally folding the web in half.  
More accurately, this means that each arc or 
half arc on the right side of web diagram is doubled in
a non-intersecting way.
The folded web is obtained by smoothing out this folded picture.
For example, if do this
to $\web_T$ above, we get
  \begin{align*}
	\vc{\includegraphics[scale=0.225]{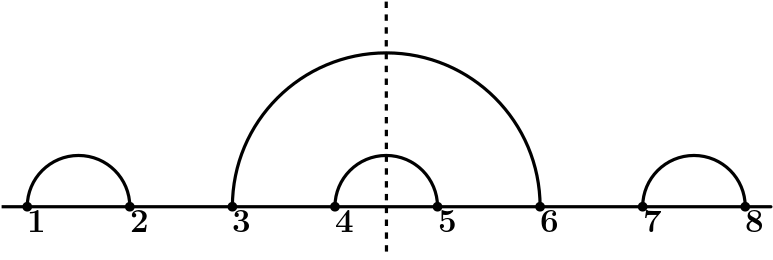}}
 \qquad & \longrightarrow \qquad 
	\vc{\includegraphics[scale=0.225]{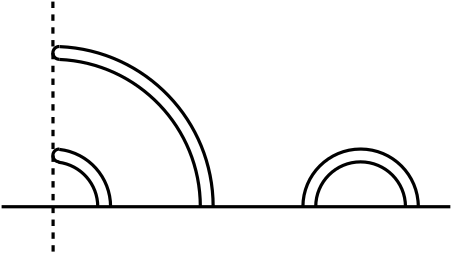}} \\
	& \longrightarrow \qquad
	\vc{\includegraphics[scale=0.225]{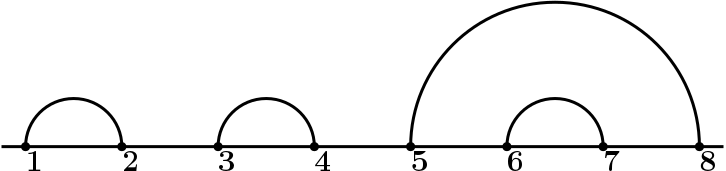}}
\,.
  \end{align*}
The resulting diagram is the $2$-web of  
\[ 
 \fold(T) = \begin{young}[c] 
1\ynobottom & 3\ynobottom & ]=5 & =]6 \\ ]=]2\ynotop&4\ynotop&]=7&8 
\end{young}
\,,
\]
the tableau we obtain by folding $T$ (see Example~\ref{ex:folding}).  
Note that
$\fold(T)$ is indeed a domino tableau: the entries $2j-1$, $2j$ are 
adjacent for all $j= 1,2,3,4$.

 \begin{thm} \label{thm:2-by-n}
 If $T$ is a rotationally symmetric $2 \times n$ standard Young tableau, then the $2$-web of $\fold(T)$ is obtained by  folding the $2$-web  of $T$ to the right with respect to its axis of symmetry.  
 \end{thm} 

A $3$-web is a kind of bipartite planar graph in which the external vertices 
have degree $1$, 
and the internal vertices have degree $3$.
The bijection between $3 \times n$ tableaux and $3$-webs is 
somewhat more complicated than the $2 \times n$ case, 
and the first algorithm for computing it was given in \cite{kuperberg}.
Tymoczko \cite{tymoczko} subsequently found 
a simpler algorithm for producing this bijection.
It begins by constructing two non-crossing arc diagrams --- one
for the first pair of rows, and one for the second pair. The two
diagrams are then superimposed to form the \emph{m-diagram} of the tableau,
which does have crossings.  The
crossings in the m-diagram are then resolved in a canonical way to 
produce the web.  
We recall the precise definition of a $3$-web and the statement of
this algorithm
in Section \ref{sec:webs}.

\begin{figure}
\centering
$T= \begin{young}[c] 
1 & 2 & 3 & 4 & 8 & 10 \\
5& 6 & 7 & 12 & 13& 14 \\
9 & 11 & 15 & 16& 17 & 18\end{young}$
\qquad
   $\vc{\includegraphics[scale = 0.32]{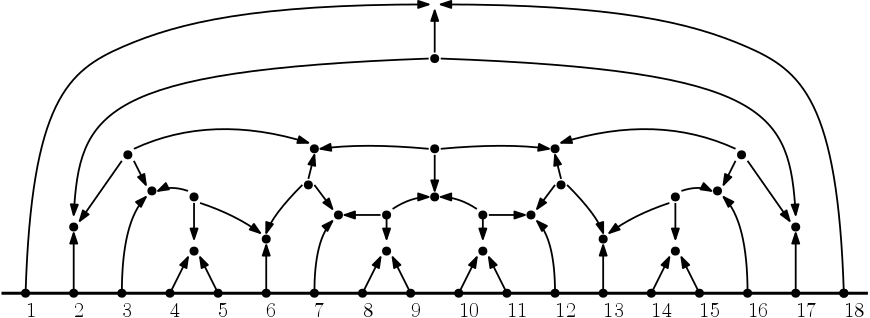}}$
\caption{A $3 \times n$ tableau $T$ and the corresponding $3$-web
$\web_T$.}
\label{fig:runningexample}
\end{figure}
Figure \ref{fig:runningexample} shows an example of a $3 \times n$
rectangular tableau and the corresponding $3$-web.
The tableau here is again rotationally symmetrical, which implies
that the corresponding $3$-web has vertical axis of symmetry.  

In light of Theorem~\ref{thm:2-by-n}, one might guess that 
for a rotationally symmetric $3 \times n$ tableau $T$,
there is some sense in
which the corresponding $3$-web $\web_T$ 
can be folded in half to produce the web of the folded tableau
$\fold(T)$.  Unfortunately,
this turns out to be incorrect.  
There does not appear to be any canonical way to geometrically fold a
$3$-web in half, and deform it so that the resulting diagram is a $3$-web,
and in general the web of $\fold(T)$ cannot be drawn in such a way that it
looks like a doubled version of the right half of $\web_T$.
Nevertheless, we show that there is still a strong a sense in which 
$\fold(T)$ corresponds to $\web_T$ modulo vertical symmetry.

Given a web $\web$, there is a simple algorithm for computing the 
corresponding tableau $T_\web$, which begins by computing the distance from 
certain faces to the outer face.
Our first main result 
is a modification of this algorithm.
Given a symmetrical web $\web$, instead of computing the distance
to the outer face, our algorithm begins by computing the distance
from a face to its mirror image.  These distances are then used to
produce a domino tableau $\domtab_\web$.

\begin{thm} \label{thm:fw1}
%
If $\web$ is a symmetrical $3$-web, then 
$\domtab_\web  = \fold(\symtab_\web)$.
\end{thm}

Our second main result is an algorithm which provides
the inverse bijection.
This is a modification of Tymoczko's algorithm.
Starting from a domino tableau $D$, we
construct a web $\webcross_\domtab$, called the \emph{crossed web}
of $D$.

Roughly, the crossed web is constructed as follows.  First, we form a
tableau $\compresstab$, called the \emph{compression} of $\domtab$,
which looks superficially similar 
to $\domtab$ but is only is half width.
We then form \emph{crossed m-diagram} 
of $\domtab$, denoted $\mdiagcross_\domtab$, 
which is obtained by ``doubling and unfolding'' 
the m-diagram of $\compresstab$. Each arc of $\compresstab$
becomes to two arcs which are mirror images of each other;
these are either two mirror-image copies of the original arc, 
or a pair of arcs that cross.
In general, the crossed m-diagram is not the m-diagram of any tableau.
Nevertheless, it is similar enough that we can take its canonical resolution.
The crossed web of $D$ is the canonical resolution of 
the crossed m-diagram of $D$.

\begin{thm}\label{thm:fw2}
If $\domtab = F(T)$ is a domino tableau, then $\webcross_\domtab = \web_T$.
\end{thm}

Hence, 
$\web \mapsto \domtab_\web$ and $\domtab \mapsto \webcross_\domtab$
are mutually inverse bijections between symmetrical webs and domino tableaux.


\section{Tableau operations}

In this section, we recall some of the basic definitions from tableau
theory, as well as the definitions and main properties of 
promotion, evacuation and folding.

A \emph{partition}
$\lambda = (\lambda_1, \dots, \lambda_l)$ is a weakly decreasing sequence
of positive integers.  We write $|\lambda| = \lambda_1 + \dots + \lambda_l$.
We identify $\lambda$ with its diagram, which is an array of 
$|\lambda|$ boxes, such
that there are $\lambda_i$ left-justified boxes in row $i$.  If 
$\mu = (\mu_1, \dots, \mu_m)$ is another partition, we write 
$\mu \leq \lambda$ if $m \leq l$ and $\mu_i \leq \lambda_i$ 
for $i=1,\dots, m$.  If $\mu \leq \lambda$,
the \emph{skew shape} $\lambda/\mu$ consists of boxes in
the diagram of $\lambda$ that are not in the diagram of $\mu$.
We write $|\lambda/\mu|  = |\lambda| - |\mu|$.  If $\mu= \varnothing$
is the empty partition, then $\lambda/\mu = \lambda$, and we
say that $\lambda$ is a \emph{straight shape}.

A \emph{standard Young tableau} of shape $\lambda/\mu$ 
is a filling of the
boxes of the diagram $\lambda/\mu$ with the numbers 
$1, \dots, N$, $N = |\lambda/\mu|$, such that the entries increase along 
rows and columns.  We denote the set of all such standard Young tableaux
by $\SYT(\lambda/\mu)$.  
We will sometimes encode a tableau 
$T \in \SYT(\lambda/\mu)$, by its \emph{row-index word} 
$w(T) = w_1 w_2 \cdots w_N$,
\[
   w_j = r, \quad  \text{if $j$ is in row $r$ of $T$}
\,.
\] 
For straight shapes, $w(T)$ uniquely specifies $T$; otherwise,
we also need to specify the skew shape.  
For example, the tableau $T$ in Figure~\ref{fig:runningexample} has
$w(T) = 111122213132223333$, and is in $\SYT(6,6,6)$.

\ysetshade{violet!30}
Given a standard Young tableau $T \in \SYT(\lambda/\mu)$ and a corner 
$\begin{young}[2ex] !\times \end{young}$
of $\mu$, we can perform a \emph{slide} as follows.  Starting
from $\begin{young}[2ex] !\times \end{young}$, draw the \emph{sliding path} 
through the entries of $T$, which moves either right or down at each step;
when there is a choice, the path moves toward  the smaller entry.  Then
shift the numbers one position forward along the path.  Here is an 
example, with $\lambda = (5,5,4,4)$, $\mu = (3,1)$.
\[
\begin{young}[c]
 , & , & , & 1 & 9  \\
 !\times & 2 & 3 & 11 & 12 \\
 4 & 6 & 7 & 13 \\
 5 & 8 & 10 & 14
\end{young}
\quad \to\quad
\begin{young}[c]
 , & , & , & 1 & 9  \\
 !\times & !2 & !3 & 11 & 12 \\
 4 & 6 & !7 & 13  \\
 5 & 8 & !10 & !14
\end{young}
\quad \to\quad
\begin{young}[c]
 , & , & , & 1 & 9   \\
 !2& !3 & !7 & 11 & 12 \\
 4 & 6 & !10 & 13  \\
 5 & 8 & !14  
\end{young}
\]
The \emph{rectification} of $T$, denoted $\rect(T)$, is the tableau 
obtained by performing slides until we have a straight shape.
This is independent of the order in which the slides are performed.

If $T$ is a tableau, and $k$ is an integer,
we will write  $T+k$ for be the tableau obtained from $T$ by adding $k$ to
all of its entries, and $T-k = T + (-k)$.   We write $T_{> k}$
and $T_{\leq k}$ for the subtableaux
formed by the entries greater than $k$ and less than or equal to $k$,
respectively.

For the remainder of this section, we will assume $\lambda$ is a partition 
and $N = |\lambda|$.

The \emph{promotion} operator $P: \SYT(\lambda) \to \SYT(\lambda)$
maps a tableau $T$ to be the unique tableau $P(T)$ such that
\[
P(T)_{\leq N-1} = \rect(T_{> 1} -1)
\,.
\]
In other words, to compute $P(T)$, we first delete the entry $1$ from $T$,
then subtract $1$ from all remaining entries, rectify, and finally
add a box containing $N$ to the result to produce a new standard 
Young tableau of shape $\lambda$.  For example,
\[
T = 
\begin{young}[c]
1 & 2 & 5  \\
3 & 4 & 8  \\
6 & 7 & 9\\
\end{young}
\ \to \ %
\begin{young}[c]
!\times & 1 & 4  \\
2 & 3 & 7  \\
5 & 6 & 8\\
\end{young}
\ \to \ %
\begin{young}[c]
!\times & !1 & 4  \\
2 & !3 & 7  \\
5 & !6 & !8\\
\end{young}
\ \to \ %
\begin{young}[c]
!1 & !3 & 4  \\
2 & !6 & 7  \\
5 & !8 \\
\end{young}
\ \to \ %
\begin{young}[c]
1 & 3 & 4  \\
2 & 6 & 7  \\
5 & 8 & 9\\
\end{young}
= P(T)
\,.
\]

For $k=1, \dots, N$, let
$P_k : \SYT(\lambda) \to \SYT(\lambda)$ be the operator
which 
performs promotion on the entries of less than or
equal to $k$, and fixes entries greater than $k$.  
Thus $P_k(T)$ is characterized by 
$P_k(T)_{> k} = T_{>k}$ and $P_k(T)_{\leq k} = P(T_{\leq k})$.
In particular,
$P_N = P$, and $P_1$ is the identity operator.

The \emph{evacuation} operator $E: \SYT(\lambda) \to \SYT(\lambda)$
is defined to be 
\[
E = P_1 \circ P_2 \circ \dots \circ P_N
\,.
\]
We think of this as repeatedly performing promotion, but
fixing the next largest entry after each iteration. 
Evacuation is an involution. 
A tableau $T$ is called \emph{self-evacuating} if $E(T) = T$.

For each number $k = 1, \dots, N$, the \emph{complement} of $k$ is the 
number $\compl{k} = N+1 -k$.  For example, $\compl{1} = N$.  We will make 
frequent use of this notation.

In the case where $\lambda$ is a rectangle, 
promotion and evacuation have the following
special properties.
\begin{thm}\label{p-and-e-on-rect-tableaux}
If $\lambda$ is a $l \times n$ rectangle then:
\begin{enumerate}
\item[(i)] $P^N(T)= T$.
\item[(ii)] $E(T)$ is obtained from $T$ by rotating $180^\circ$ and 
replacing each entry $k$ with its complement $\compl{k}$.
\end{enumerate}
\end{thm}

In light of (ii), we say that a self-evacuating rectangular tableau
is \emph{rotationally symmetric}.

The definition of folding is similar to evacuation, except that we fix
the next \emph{two} largest entries after each iteration of promotion.
For $j=1, \dots, \fhalfN$, define \emph{partial folding}
operators $f^j : \SYT(\lambda) \to \SYT(\lambda)$, by
\[
   f^j = P_{N-2j+2} \circ P_{N-2j+4} \circ \dots \circ P_{N}
\]
The \emph{folding} operator $F: \SYT(\lambda) \to \SYT(\lambda)$ is
$F = f^{\fhalfN}$.

\begin{ex}  
\label{ex:folding}
\ysetshade{blue!30}
If we fold the tableau
  \[
T =   
\begin{young}[c]
   1 & 3 & 4 & 7 \\
   2 & 5 & 6 & 8 \\
  \end{young} 
\]
we obtain
\[
T 
\xmapsto{P_8}
\begin{young}[c]
   1 & 2 & 3 & 6 \\
   4 & 5 & ]=!7 & !8 \\
  \end{young} 
\xmapsto{P_6}
\begin{young}[c]
   1 & 2 & ]=!5 & !6 \\
   ]=]3 & 4 & ]=!7 & !8 \\
  \end{young} 
\xmapsto{P_4}
\begin{young}[c]
   1 & !3\ynobottom & ]=!5 & !6 \\
   ]=]2 & !4\ynotop & ]=!7 & !8 \\
  \end{young} 
\xmapsto{P_2}
\begin{young}[c]
   !1\ynobottom & !3\ynobottom & ]=!5 & !6 \\
   ]=]!2\ynotop & !4\ynotop & ]=!7 & !8 \\
  \end{young} 
 = \fold(T)
\,.
\]
The shaded entries are fixed for following promotion operation in sequence.
\end{ex}

In the example above, the pairs of $(1,2)$, $(3,4)$, $(5,6)$ and
$(7,8)$ are adjacent
in $\fold(T)$.  We have emphasized this visually by removing the horizontal or 
vertical line between these entries. 
If $N$ is even, we say that $T \in \SYT(\lambda)$ is
a \emph{domino tableau} if entries $2j-1$ and $2j$ are adjacent in $T$, for
$j = 1,\dots, \tfrac{N}{2}$.  Each pair of adjacent entries $(2j-1, 2j)$ is
called a \emph{domino}.
For $N$ odd, we will say that $T \in \SYT(\lambda)$ is a domino tableau
if entries $2j$ and $2j+1$ are adjacent, for $j = 1,\dots, \tfrac{N-1}{2}$.
In the odd case, the lone entry $1$ is not part of a domino.
In both cases, we have the following connection between evacuation and folding.

\begin{prop}
$T \in \SYT(\lambda)$ is self-evacuating if and only if $\fold(T)$ is 
a domino tableau.  In particular, if $\lambda$ is a 
rectangle, $T$ is rotationally symmetric if and only if
$\fold(T)$ is a domino tableau.
\end{prop}

We conclude this section by recording three additional observations about 
promotion and folding that are
used in the proof of Theorems~\ref{thm:2-by-n} and \ref{thm:fw1}.

\begin{lemma} \label{promotion-rectification}
Suppose $\lambda$ is a rectangle, and $T \in \SYT(\lambda)$. Let $k$ be an element of $ 0, \hdots, N$, and let
$T' = P^k(T)$.  Then $T'_{\leq N-k} = \rect(T_{> k} - k)$, and
$T_{\leq k} = \rect(T'_{> N-k} - (N-k))$.
\end{lemma}

\begin{proof}
The first statement follows almost by definition (see e.g.
\cite[Ch. 2, Lemma 3]{fulton}).  Since $T$ is rectangular, $P^N(T) = T$.
Thus $T = P^{N-k}(T')$, and the second statement follows by a symmetrical 
argument.
\end{proof}

\begin{lemma} 
\label{lem:promotion-folding}
For $T \in \SYT(\lambda)$ and $j \in \{1, \dots, \fhalfN\}$, 
we have $f^j(T)_{\leq \compl{2j}} = P^j(T)_{\leq \compl{2j}}$.
In particular the entry $\compl{2j}$ in $\fold(T)$ is the same box 
as it is in $P^j(T)$.
\end{lemma}

\begin{proof}
The statement is true for $j=1$, since $f^1 = P$.   Assume $j >1$
and the result is true for $j-1$.  Let $k = \compl{2j}$.
By definition, $f^j(T) = P_{\compl{2j}+1} \circ f^{j-1}(T)$.  Therefore,
$f^j(T)_{\leq k} = 
(P_{k+1}  \circ f^{j-1}(T))_{\leq k}
= P_{k+1} (f^{j-1}(T)_{\leq k+2})_{\leq k}
= P_{k+1} (P^{j-1}(T)_{\leq k+2})_{\leq k}
= (P_{k+1} \circ P^{j-1}(T))_{\leq k}
= P^j(T)_{\leq k}$.  This proves the first statement.
The second statement follows, since the entry $\compl{2j}$ is in the
same box in $\fold(T)$ as it is in $f^j(T)$.
\end{proof}

\begin{lemma}
\label{lem:promotion-folding-symmetric}
Suppose $T \in \SYT(\lambda)$ is self-evacuating. Then for 
$j \in \{1, \dots, \chalfN\}$, 
the entry $\compl{2j}+1$ in $\fold(T)$ is the same box 
as it is in $P^{j-1}(T)$.
\end{lemma}

\begin{proof}
Let $T' = f^{j-1}(T)$.  Since $F(T)$ is a domino tableau, entries 
$\compl{2j}$ and $\compl{2j}+1$
form a domino in $f^j(T) = P_{\compl{2j}+1}(T')$.  This
implies that the sliding path in the 
application of $P_{\compl{2j}+1}$ to $T'$ must end at $\compl{2j}+1$.  
Thus $\compl{2j}+1$ is in the same box in $f^{j}(T)$ as it is in 
$f^{j-1}(T)$.  The result now follows as in
Lemma~\ref{lem:promotion-folding}.
\end{proof}


\section{Webs}
\label{sec:webs}
In this section, we recall the definitions of $2$-webs and $3$-webs, as
well as the bijections between webs and rectangular tableaux. 

\begin{defn}
\label{def:arc-diagram}
An \emph{arc diagram} is a drawing of a graph in the plane, in which
there are $N$ vertices (usually labelled $1, \dots, N$, from left to right),
along a horizontal line called the \emph{boundary line}.  
We have arcs joining some pairs of vertices, which are above
the boundary line.  The arcs may be directed or undirected.  We
denote an undirected arc joining $a$ and $b$ as $\{a,b\}$ and
a directed arc from $a$ to $b$ as $(a,b)$.  We require that
each pair of arcs have at most one point of intersection 
(including endpoints), and no three arcs are concurrent (excluding endpoints).
We say that two arcs \emph{cross} each other if they intersect not at an endpoint. In other words,  $(a,b)$ and $(c,d)$ (or $\{a,b\}$ and $\{c,d\}$) \emph{cross}
if and only if
$\min{a,b} < \min{c,d} < \max{a,b} < \max{c,d}$ or vice-versa.
\end{defn}

Note that an arc diagram is determined (up to planar isotopy) by its
underlying graph if and only if there are no three arcs that cross
pairwise.

If $X$ is a subset of the plane, and $\alpha$ is an arc in an arc diagram, 
we say that $\alpha$ is \emph{above} $X$, if $X$ is contained in the 
closed region bounded by $\alpha$ and the boundary line.  
If $(a,b)$ and $(c,d)$ (or $\{a,b\}$ and $\{c,d\}$) are 
two distinct arcs in an arc diagram, then $(a,b)$ is \emph{above} $(c,d)$ 
if and only if
$\min{a,b} < \min{c,d} < \max{c,d} < \max{a,b}$.   This defines a
partial order on the arcs.

\begin{defn}
A \emph{$2$-web} is an undirected arc diagram with $N = 2n$
vertices, in which every vertex has degree $1$ 
and there are no crossings.
\end{defn}

As noted in the introduction, there is a simple bijection between
$\SYT(n,n)$ and the set of $2$-webs on $2n$ vertices. A
tableau $T \in \SYT(n,n)$ corresponds to the unique web $\web_T$
such that that entries
in the first row of $T$ are the left ends of arcs in $\web_T$, and 
the entries in the second row of $T$ are the right ends of arcs.

Webs can be cyclically rotated and reflected.  For a web $\web$,
define $P(\web)$ to be the cyclic rotation of $\web$.  This is
the unique web such that
the map $k \mapsto k-1 \pmod N$ induces a graph isomorphism 
from $\web$ to $P(\web)$.  
This is most easily visualized by drawing the boundary
line as a circle, instead of as a straight line.  
For example, suppose $\web$ is the $2$-web below.
\[
    \includegraphics[scale=0.225]{arc_diag.png}
\]
If we draw the $2$-webs $\web$ and $P(\web)$ on a circle, 
we get the following drawings.
\[ 
 \web \includegraphics[scale=0.2]{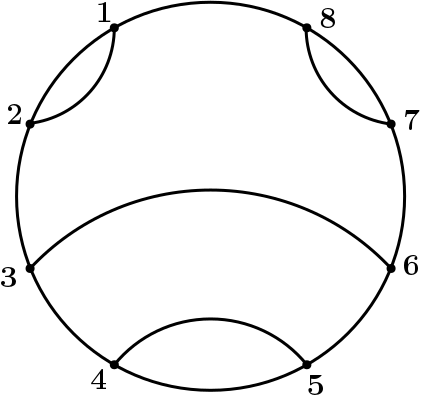} \qquad \qquad \qquad
 P(\web) \includegraphics[scale=0.2]{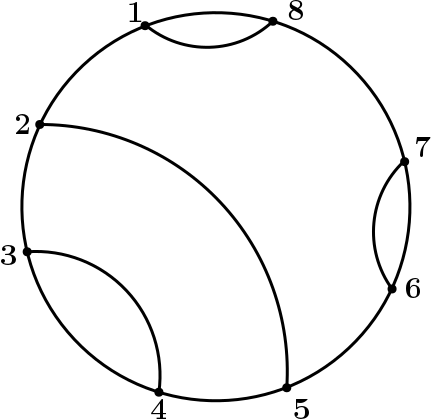} 
\]
In this picture, $P$ fixes the labels, and rotates the interior of the
circle one step clockwise.
We define $E(\web)$ to be the reflection of $\web$;
here the map $k \mapsto \compl{k}$ induces a graph isomorphism from
$\web$ to $E(\web)$.  We will say that a web is \emph{symmetrical} if
$E(\web) = \web$.   The operations $P$ and $E$ on webs correspond
to promotion and evacuation on tableaux.

\begin{thm}
If $T \in \SYT(n,n)$ then
$P(\web_T) = \web_{P(T)}$ and
$E(\web_T) = \web_{E(T)}$.
\end{thm}

\begin{proof}
The first assertion is discussed in \cite[Theorem 1.4]{PPR}, and the 
second is in \cite[Theorem 3.3]{PaPe}.
\end{proof}

If $\web$ is a symmetrical $2$-web, we define $\fold(\web)$ to be the
web obtained by ``folding $\web$ to the right''.  Formally this means:
\begin{itemize}
\item if $\{a,b\}$ and $\{\compl{a}, \compl{b}\}$ are arcs of $\web$,
where $a < b \leq n$, then
$\fold(\web)$ has corresponding arcs $\{\compl{2a}, \compl{2b}+1\}$ and 
$\{\compl{2a}+1, \compl{2b}\}$;
\item if $\{a,\compl{a}\}$ is an arc of $\web$, $a \leq n$, then 
$\fold(\web)$ has a corresponding arc $\{\compl{2a}, \compl{2a}+1\}$.
\end{itemize}
Theorem~\ref{thm:2-by-n} states that if $T \in \SYT(n,n)$ is rotationally
symmetrical, then $\web_{\fold(T)} = \fold(\web_T)$.

\begin{defn} 
\label{def:3-web}
A \emph{$3$-web} is directed bipartite graph drawn in the plane, satisfying the 
following conditions.
\begin{enumerate}
    \item[(1)] There are $N = 3n$ \emph{boundary vertices} embedded on a horizontal \emph{boundary line}. 
These are labelled from $1, \dots, N$ from left to right. 
All boundary vertices have degree $1$ and are sources. 
    \item[(2)] The \emph{internal vertices} (vertices that are not on the boundary line) have degree $3$, and are either sources or sinks. 
    \item[(3)] All edges are above the boundary line. 
    \item[(4)] All \emph{internal faces} (faces that are not incident with the boundary line)
       have at least $6$ sides. 
\end{enumerate}
\end{defn} 

In addition to the internal faces, a $3$-web $\web$ also has \emph{boundary
faces}, which are bounded by arcs and the boundary line.
We denote the boundary face incident with boundary vertices 
$k$ and $k+1$ by $B_k$; for the outer face 
(which is incident with boundary vertices $1$ and $3n$) we write 
$B_0 = B_N$.
If $X$ and $Y$ are any two faces in a $3$-web $\web$, let 
$\webdist(X, Y)$ denote the distance between $X$ and $Y$,
i.e. the minimal number of edges we need to cross to get from 
$X$ to $Y$. This is called the \emph{web-distance}. 

We now recall the algorithms defining a bijection between $3$-webs
on $3n$ boundary vertices and $\SYT(n,n,n)$.


\begin{algo} \label{alg:tw1}
\setlength{\parindent}{0pt}
{\it The tableau $T_\web$ associated to a $3$-web $\web$.}

\medskip

{\tt Input:} Let $\web$ be a $3$-web on $N =3n$ boundary vertices.

\medskip

For $z \in \{-1,0,1\}$, define
\[
\Phi(z) = \begin{cases} 
1  & \text{if } z = -1\\
2 & \text{if } z = 0 \\
3 & \text{if } z = 1 .
\end{cases} 
\]
For $i=0, \dots, N$, let $d_i := \webdist(B_0, B_i)$. \\
For $i=1, \dots, N$, let $w_i := \Phi(d_{i-1} - d_i)$. 
(Note that since $B_{i-1}$ and $B_{i}$ are adjacent, we necessarily
have $d_i - d_{i-1} \in \{-1,0,1\}$, so this is defined.)

\medskip

{\tt Output:} 
The unique tableau $T_\web \in \SYT(n,n,n)$ such that
$w(T_\web) = w_1 w_2 \dots w_N$.
\end{algo}

\begin{ex}
For the web $\web$ in Figure~\ref{fig:runningexample},
the distances $d_i$ are labelled on the faces $B_i$ as shown below.
\begin{center}
    \includegraphics[scale = 0.5]{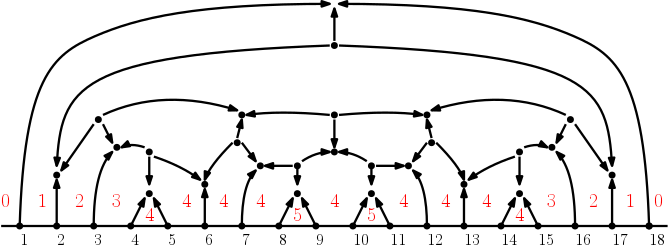}
\end{center}
Thus, Algorithm~\ref{alg:tw1} produces $w(T_\web) = 111122213132223333$,
which is the row-index word of the tableau shown 
in Figure~\ref{fig:runningexample}.
\end{ex}

We now recall the algorithm for the inverse map, as given 
in \cite{tymoczko}.

\begin{algo} \label{alg:tw2} 
{\it The web $\web_T$ associated to tableau $T$.}

\medskip

\noindent {\tt Input:} A tableau $T \in \SYT(n,n,n)$.

\medskip

\noindent {\tt Step 1:} Drawing the m-diagram.

Begin by drawing the boundary line and boundary points labelled 
$1, \dots, 3n$.  We now construct an arc diagram called $\mdiag_T$, 
the \emph{m-diagram} of $T$ as follows.

Draw the $2$-web for the $2 \times n$ tableau formed by the top
two rows of $T$.  Call these arcs \emph{first arcs}, and direct
the first arcs so that they are pointing from left to right. 

Draw the $2$-web for the $2 \times n$ tableau formed by the second
two rows of $T$. Call these arcs \emph{second arcs}, and direct
the second arcs so that they are pointing from right to left.

\medskip

\noindent{\tt Step 2:} Resolving the m-diagram.

Each boundary vertex of the m-diagram is either a source of degree $1$
or a sink of degree $2$.  For
each vertex of the second type, make the following local change
 \begin{center}
   \vc{\includegraphics[scale =0.4]{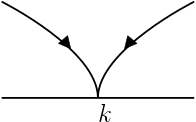}} $\quad \to \quad$ \vc{\includegraphics[scale =0.4]{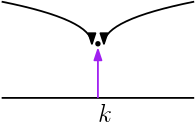}} 
 \end{center}
so that every boundary vertex is now a source of degree $1$.

If two internal arcs cross, make the following local change
 \begin{center}
   \vc{\includegraphics[scale =0.4]{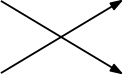}} $\quad \to \quad$ \vc{\includegraphics[scale =0.4]{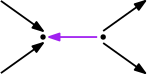}} 
 \end{center}

 The new edges that arises from the local changes (marked in purple)    are called \emph{intersection edges},  and the original edges (marked in black)  are called \emph{arc edges}.  

\medskip

\noindent{\tt Output:} The resulting planar graph $\web_T$.
\end{algo}

\begin{ex}
If we apply Algorithm~\ref{alg:tw2} to the tableau $T$
in Figure~\ref{fig:runningexample}, Step 1 produces the m-diagram
below. The red arcs are the first arcs and the blue arcs are the second arcs. 
\begin{center}
    \includegraphics[scale = 0.4]{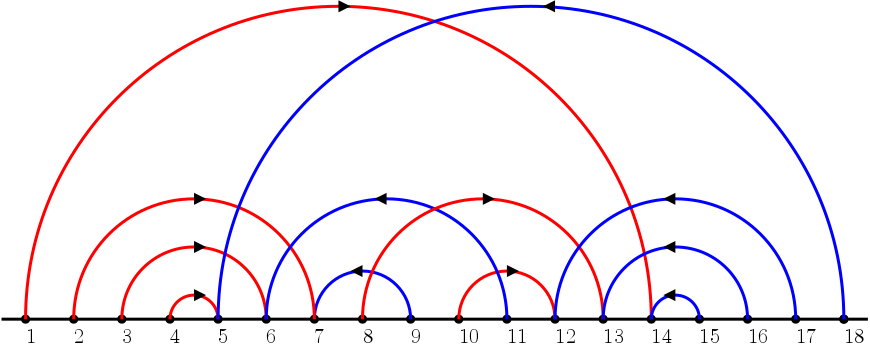}
\end{center}
Taking the canonical resolution, in Step 2, we obtain the web $\web_T$, shown
in Figure~\ref{fig:runningexample}.
\end{ex}

In order to discuss variants of Algorithm~\ref{alg:tw2}, we introduce
the following definition.

\begin{defn}
A \emph{generalized m-diagram} is a directed arc diagram with the property
that a boundary vertex is either a source of degree $1$ or a sink of degree $2$.
\end{defn}

Any generalized m-diagram can be \emph{resolved}
into a bipartite planar graph, by the process
described in Step 2 of Algorithm~\ref{alg:tw2}.
However, in general, the resolution need not be a $3$-web, since 
in some cases, the 
resolution will have internal faces of degree less than $6$.

Moreover, unlike an m-diagram, a generalized m-diagram is not
necessarily determined by its underlying graph.  There may be several
inequivalent ways (up to planar isotopy) to draw a generalized 
m-diagram, and these can have inequivalent resolutions.  

Fortunately, the two difficulties above are related.

\begin{prop}
\label{prop:independent-of-drawing}
Suppose $\mdiag$ is a generalized m-diagram whose resolution is a $3$-web
$\web$.  Then every drawing of $\mdiag$ as a valid arc diagram has
$\web$ as its resolution.
\end{prop}

\begin{proof}
Any two drawings of $\mdiag$ are related by a sequence of 
type III shadow Reidemeister moves.
In order to avoid creating faces of degree $4$ in the resolution, 
the orientation of the arcs involved in the move must be as shown below.
\[
  \vc{\includegraphics[scale=0.3]{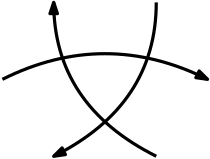}}
  \qquad \quad \longleftrightarrow \quad \qquad
  \vc{\includegraphics[scale=0.3]{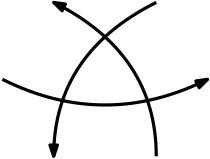}}
\]
However, both of the local diagrams above have the same resolution.
Thus the resolution is unaffected by any sequence of such moves.
\end{proof}

Cyclic rotation and reflection of $3$-webs is defined exactly the same
way as for $2$-webs.  We denote cyclic rotation map by $P$, and the
reflection map by $E$.  Again, these correspond to promotion and
evacuation on tableaux.

\begin{thm}
If $T \in \SYT(n,n,n)$, then
$P(\web_T) = \web_{P(T)}$ and
$E(\web_T) = \web_{E(T)}$.
\end{thm}

\begin{proof}
The first assertion is discussed in \cite[Theorem 2.5]{PPR},
and the second is proved in \cite[Theorem 3.3]{PaPe}.
\end{proof}


\section{Algorithms for symmetrical $3$-webs}

In this section, we modify Algorithms~\ref{alg:tw1} and~\ref{alg:tw2} to
produce bijections between symmetrical webs and domino tableaux.
We first modify Algorithm~\ref{alg:tw1} so that instead of computing
the distance from a face $B_j$ to the outer face $B_0$, we compute
the distance from $B_j$ to its mirror image $B_{N-j}$.

\begin{algo} \label{alg:fw1}
{\it The domino tableau $\domtab_\web$ associated to a symmetrical $3$-web $\web$.}

\medskip

\noindent{\tt Input:} A symmetrical $3$-web $\web$ on $N = 3n$ boundary vertices.

\medskip

For $z \in \{-2,-1,0,1,2\}$, define
\[  \Lambda (z)  
= \begin{cases} (1,1)  & \text{if } z = -2 \\
(1,2) &\text{if } z = -1 \\ 
(2,2) & \text{if } z = 0 \\
(2,3) & \text{if } z = 1 \\
(3,3) & \text{if } z = 2. \\
\end{cases}
\,.
 \] 

 For $j=0, \dots, \fhalfN$, let $h_j := \webdist(B_j, B_{N-j})$.  

\smallskip

For $j=1, \dots, \fhalfN$, let 
$(w_{\compl{2j}}  , w_{\compl{2j}+1} )  := \Lambda(h_j - h_{j-1})$. 

\smallskip

If $N$ is odd, let $w_1 :=1$.

\medskip

{\tt Output:} 
The unique tableau $\domtab_\web \in \SYT(n,n,n)$ such that
$w(\domtab) = w_1 w_2 \dots w_N$.
\end{algo}

\begin{ex} \label{ex:fw1}
For the $3$-web for $\web$ in Figure~\ref{fig:runningexample},
the distances $h_j$ are labelled on the faces $B_j$ as shown below.
 \begin{center}
    \includegraphics[scale = 0.4]{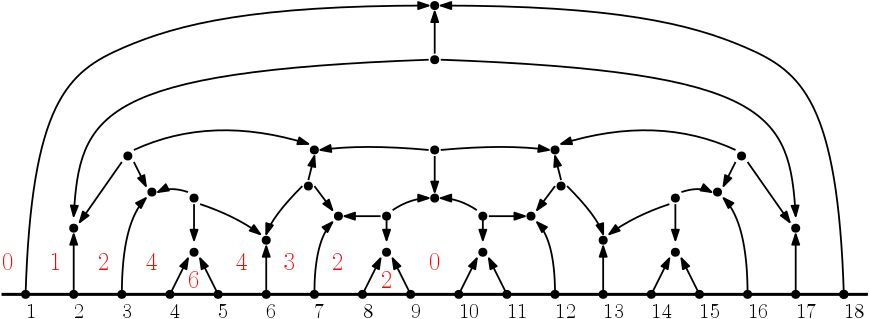}
\end{center}
Now we compute:
\begin{align*}
(w_{1}, w_{2}) &= \Lambda(h_9-h_8) = \Lambda (0-2) = (1,1) \\
(w_{3}, w_{4}) &= \Lambda(h_8-h_7) = \Lambda (2-2) = (2,2) \\
(w_{5}, w_{6}) &= \Lambda(h_7-h_6) = \Lambda (2-3) = (1,2) \\
(w_{7}, w_{8}) &= \Lambda(h_6-h_5) = \Lambda (3-4) = (1,2) \\
(w_{9}, w_{10}) &= \Lambda(h_5-h_4) = \Lambda (4-6) = (1,1)   \\
(w_{11}, w_{12}) &= \Lambda(h_4-h_3) = \Lambda (6-4) = (3,3)  \\
(w_{13}, w_{14}) &= \Lambda(h_3-h_2) = \Lambda (4-2) = (3,3) \\
(w_{15}, w_{16}) &= \Lambda(h_2-h_1) = \Lambda (2-1) = (2,3) \\
(w_{17}, w_{18}) &= \Lambda(h_1-h_0) = \Lambda (1-0) = (2,3). 
\end{align*}
Thus $w(\domtab_\web) = 112212121133332323$, and
\[
\domtab_\web = 
   \begin{young}[c]
    ]=1 & =]2 & ]=]5\ynobottom & ]=]7\ynobottom & ]=9 & =]10 \\
    ]=3 & =]4 & ]=]6\ynotop & ]=]8\ynotop & ]=]15\ynobottom & =]17\ynobottom \\
    ]=11 & =]12 & ]=13 & =]14 & ]=]16\ynotop & =]18\ynotop \\
   \end{young}
\,.
\]
One can verify directly that $\domtab_\web = \fold(T_\web)$.
\end{ex}


We now give the inverse algorithm, which is a modification of
Algorithm~\ref{alg:tw2}.
There are enough minor differences between the even and odd cases that it
is easier to  treat them separately.  We begin with the case where $n$ is even.

For $n$ even, a $3 \times n$ domino tableau is a concatenation
of \emph{simple blocks} of the following shapes.
\begin{center}
\begin{tabular}{ccc}
\includegraphics[scale=0.25]{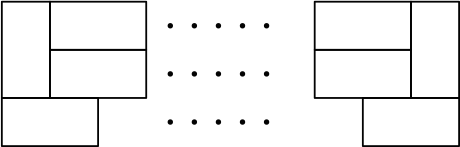}  &
~~~~~\includegraphics[scale=0.25]{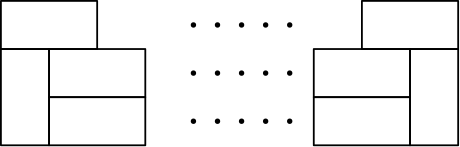}~~~~~  &
\includegraphics[scale=0.25]{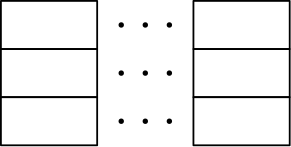}  \\
Type 1 &  Type 2 & Type 3 \\
\end{tabular}
\end{center}
Here, the omitted dominoes are always horizontal.
Blocks of types 1 and 2 have two vertical dominoes, and type 3 blocks have 
no vertical dominoes.
For example, the tableau $\domtab_\web$ in Example~\ref{ex:fw1} is
the concatenation of the blocks 
\[
   \begin{young}[c]
    ]=1 & =]2  \\
    ]=3 & =]4  \\
    ]=11 & =]12 
   \end{young}
\qquad
   \begin{young}[c]
 ]=]5\ynobottom & ]=]7\ynobottom \\
 ]=]6\ynotop & ]=]8\ynotop \\
 ]=13 & =]14 
   \end{young}
\qquad\text{and}\qquad
   \begin{young}[c]
 ]=9 & =]10 \\
 ]=]15\ynobottom & =]17\ynobottom \\
 ]=]16\ynotop & =]18\ynotop \\
   \end{young}
\]
which are of type 3, type 1, and type 2, respectively. 


\begin{algo}
\label{alg:fw2-even} 
{\it The crossed web $\webcross_\domtab$ of a domino tableau $\domtab$ (even case).}

\medskip

\noindent{\tt Input:} A domino tableau $\domtab \in \SYT(n,n,n)$, 
where $n$ is even.

\medskip

\noindent{\tt Step 1:}
Let $N := 3n$, $M := \frac{N}{2}$.
For each $ k  \in \{1, \hdots , M\}$, replace the pair of
entries $ (2k-1, 2k)$ with a single label $k$.  
Call the resulting tableau $\rdomtab$. 

\medskip

\noindent{\tt Step 2:}
Decompose $\rdomtab$ into its blocks.  Each block of type 1
or type 2 has two vertical dominoes, $k_1$ and $k_2$, from which we
form a \emph{vertical domino pair} $(k_1, k_2)$.
For type 1 blocks, we take $k_1 < k_2$, and for type 2 blocks 
we take $k_1 > k_2$. 
In either case, we call $\min{k_1,k_2}$ a \emph{first vertical domino}
and $\max{k_1,k_2}$ a \emph{second vertical domino}.
Define $V$ to be the set of all vertical domino pairs in $\rdomtab$.

\medskip

\noindent{\tt Step 3:}
 For each $k \in \{1, \dots, M\}$, 
if $k$ labels a horizontal domino in row $r$, 
 then let $v_k := r$.  If $k$ labels a vertical domino in rows $r$ and $r+1$,
we let $v_k := r$ if $k$ is a first vertical domino, and $v_k := r+1$
if $k$ is a second vertical domino.  
Define $C \in \SYT(\frac{n}{2}, \frac{n}{2}, \frac{n}{2})$, 
to be the tableau with 
$w(C)= v_1v_2 \dots v_M$.  $C$ is called the
\emph{compression} of $D$.

\medskip

\noindent{\tt Step 4:}
Draw the m-diagram of $C$ on vertices labelled 
$ 1, \hdots , M$. Reflect it to the left and label the
reflected vertices $M' , \hdots, 2', 1' $, from left to right.
This arc diagram is called 
the \emph{reflected m-diagram} of $C$.

\medskip

\noindent{\tt Step 5:}
We assert that every pair in  $V$ is a directed arc in 
the m-diagram of $C$.
For each $(k_1,k_2) \in V$, remove the direct arcs 
$(k_1,k_2)$ and $(k_1', k_2')$ from the reflected m-diagram of $\compresstab$ 
and replace them with directed arcs $(k_1, k_2')$ and $(k_1', k_2)$.
The resulting arc diagram, denoted $\mdiagcross_\domtab$, is called the
\emph{crossed m-diagram} of $D$.

\medskip

\noindent{\tt Step 6:} The crossed m-diagram is a generalized m-diagram,
so it has a resolution.  Define the crossed web $\webcross_\domtab$ 
to be the 
resolution of $\mdiagcross_\domtab$.  Relabel the vertices $1, \dots, 3n$.

\medskip

\noindent{\tt Output:} %
$\webcross_\domtab$
\end{algo}

\begin{ex} 
Applying Algorithm~\ref{alg:fw2-even} to the tableau
$\domtab = \domtab_\web$ from Example~\ref{ex:fw1}, in
Steps 1--3, we obtain
\[
\rdomtab = \vc{\includegraphics[scale=0.3]{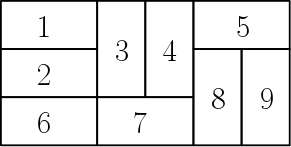}} \\
\qquad\qquad
V = \big\{ (3,4) , (9,8) \big\}
\qquad\qquad
C = \begin{young}[c] 1 & 3 & 5 \\ 2 & 4 & 8 \\ 6 & 7 & 9 \end{young}
\,.
\]
The reflected m-diagram of $C$ in Step 4 is
\[
\vc{\includegraphics[scale = 0.35]{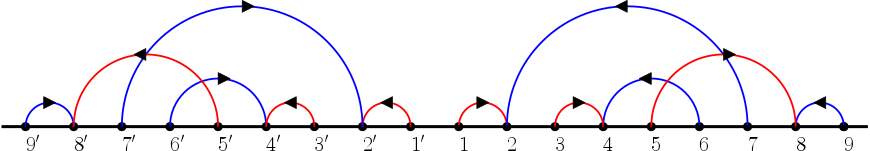}}
\]
and crossing arcs from $V$ in Step 5
yields crossed m-diagram $\mdiagcross_\domtab$.
\[
\vc{\includegraphics[scale=0.35]{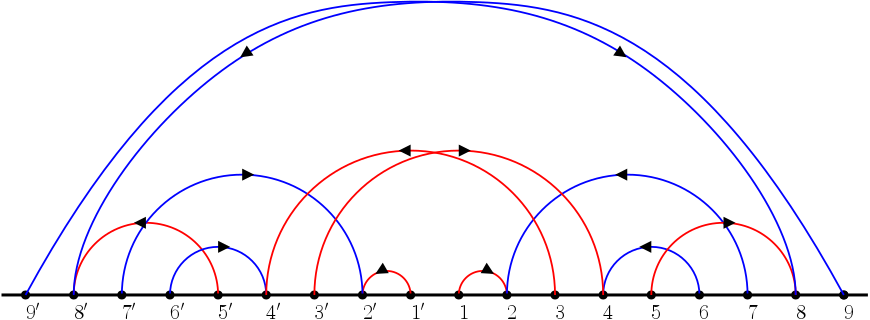}}
\]
Taking the resolution in Step 6, we once again obtain the web shown in 
Figure~\ref{fig:runningexample}.
\end{ex}

The case where $n$ is odd is similar, with a few small modifications.
If $\domtab \in \SYT(n,n,n)$ is a domino tableau, with $n$ odd, then the
entry $1$ is not part of a domino, and so the simple decomposition has 
a fourth type of block which occurs first and only occurs once.

\begin{center}
\begin{tabular}{c}
  \includegraphics[scale=0.25]{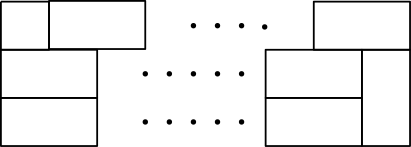}  
\\
  Type 0 
\end{tabular}
\end{center}


\begin{algo}
\label{alg:fw2-odd}
{\it The crossed web $\webcross_\domtab$ of a domino tableau $\domtab$ (odd case).}

\medskip

\noindent{\tt Input:} A domino tableau $\domtab \in \SYT(n,n,n)$, where $n$ is odd.

\medskip

\noindent{\tt Step 1:}
Let $N := 3n$, $M := \tfrac{N-1}{2}$.
First subtract $1$ from all entries of $\domtab$.
Then replace the domino labels as described in Step 1 of 
Algorithm~\ref{alg:fw2-even} to produce $\rdomtab$. 

\medskip

\noindent{\tt Step 2:}
Decompose $\domtab$ into simple blocks.  
For each type 1 and type 2 block we have a vertical domino pair
as in Algorithm~\ref{alg:fw2-even}.  In addition, we assign
the pair $(k,0)$ to the type 0 block, where $k$ is the unique 
vertical domino in this block.  For the purpose of Step 3, we regard 
$k$ as a second vertical domino.
Define $V$ as the set of all such pairs.

\medskip

\noindent{\tt Step 3:}
Construct $C$ as in Step 3 of 
Algorithm~\ref{alg:fw2-even}, except that the shape of $C$ will be 
$(\tfrac{n+1}{2},\tfrac{n+1}{2},\tfrac{n+1}{2})/(1,1)$.
Also define $C_0$ to be the skew tableau of shape
$(\tfrac{n+1}{2},\tfrac{n+1}{2},\tfrac{n+1}{2})/(1)$, obtained
by adding an entry $0$ to $C$ in row $2$, column $1$.

\medskip

\noindent{\tt Step 4:}
We draw the reflected m-diagram of $C$ as follows.
The $3n$ vertices on the boundary line will be labelled
$M', \dots, 2',1',0,1,2, \dots, M$ from left
to right.  Draw first arcs using the first two rows of $C$
on vertices $1, \dots, M$; these are oriented left to right.  
Draw the second arcs using the second two 
rows of $C_0$ on vertices $0, 1, \dots, M$; 
these are oriented right to left.
Finally, reflect every arc to the left so 
that each vertex $k \in \{1, \dots, M\}$ is reflected onto $k'$,
and $0= 0'$ is its own reflection.

\medskip


\noindent{\tt Steps 5 and 6:}
Construct the crossed m-diagram $\mdiagcross_\domtab$ and the crossed web
$\webcross_\domtab$ 
exactly as in Steps 5 and 6 of Algorithm~\ref{alg:fw2-even}.

\medskip

\noindent{\tt Output:} $\webcross_\domtab$

%
%
\end{algo}

\begin{ex}
We apply Algorithm~\ref{alg:fw2-odd} to the tableau
\[
   \domtab = \begin{young}[c]
1 & ]=2 & =]3  \\
]=]4\ynobottom & 6\ynobottom & 8\ynobottom \\
5\ynotop & 7\ynotop & 9\ynotop \\
\end{young}
\,.
\]
Steps 1--3 yield
\[
\rdomtab = \vc{\includegraphics[scale=0.3]{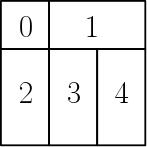}}
\qquad\quad
V = \big\{ (2,0) , (4,3)   \big\}
\qquad\quad
C =  \begin{young}[c] , & 1 \\ , & 3 \\ 2 & 4 \end{young}
\qquad\quad
C_0 =  \begin{young}[c] , & 1 \\ 0 & 3 \\ 2 & 4 \end{young}
\,.
\]
In Step 4, the reflected m-diagram of $C$ is
\[
\vc{\includegraphics[scale=0.5]{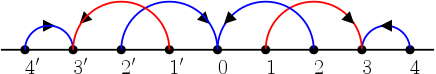}}
\]
and Steps 5 and 6 produce the diagrams $\mdiagcross_\domtab$ 
and $\webcross_\domtab$, 
shown below.
\[
\includegraphics[scale=0.45]{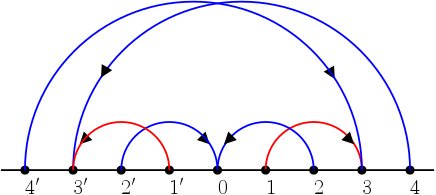}
\qquad
\includegraphics[scale=0.45]{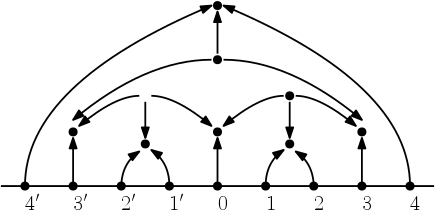}
\]
\end{ex}

\begin{remark}
Unlike an m-diagram, which is unique up to isotopy, there may be more
than one way to draw a crossed m-diagram.
However, Proposition~\ref{prop:independent-of-drawing} tells us that 
this does not matter ---
since the resolution is a $3$-web, every drawing of the crossed m-diagram 
resolves to the same $3$-web $\webcross_\domtab$.
\end{remark}


\section{Proof of  Theorem \ref{thm:2-by-n}} 

In this section, we prove Theorem~\ref{thm:2-by-n}.
Let $T$ be a $2\times n$ standard Young tableau.  
We will assume throughout that $T$ is rotationally symmetric (or equivalently
that its $2$-web $\web_T$ is symmetric).
As usual, $N = 2n$ is the number of boxes of $T$. 
We assume inductively, that the result is true for tableaux with fewer 
boxes than $T$.  It is enough to show that $\web_{\fold(T)}$ and 
$\fold(\web_T)$ have the same underlying graph.  
There are two cases.

\subsubsection*{Case 1: $\{1,\compl{1}\}$ is an arc of $\web_T$.}

Denote the remaining $n-1$ arcs of $\web_T$ by
$\{a_1, b_1\}, \dots, \{a_{n-1}, b_{n-1}\}$.
Then the arcs of $P(\web_T)$ are 
$\{a_1-1, b_1-1\},. \dots, \{a_{n-1}-1, b_{n-1}-1\},$ and $\{N-1,N\}$.
Therefore the entries $N-1$ and $N$ form a vertical domino in $P(T)$.

Let $T' = P(T)_{\leq N-2}$.  By the above, this is a $2 \times (n-1)$ 
rectangular
tableau; the corresponding $2$-web $\web_{T'}$ consists of arcs
$\{a_1-1, b_1-1\}, \dots, \{a_{n-1}-1, b_{n-1}-1\}$.  Since $\web_T$ is
symmetric, $\web_{T'}$ is symmetric. The arcs of $\fold(\web_T)$ are $\{N-1,N\}$
together with the arcs of $\fold(\web_{T'})$.

On the other hand,
\[
   \fold(T)  =  
   \begin{young}[c][6ex]
    ]= & \fold(T') & =] & {\begin{matrix} N{-}1 \\ N \end{matrix}}
   \end{young}\,,
\]
so the arcs of $\web_{\fold(T)}$ are $\{N-1,N\}$ together with the arcs
of $\web_{\fold(T')}$.  By inductive hypothesis, $\web_{\fold(T')} = \fold(\web_{T'})$,
so we conclude that $\fold(\web_T)$ and $\web_{\fold(T)}$ have the
same arcs.

\subsubsection*{Case 2: $\web_T$ has arcs $\{1, 2\ell\}$ and 
$\{\compl{1}, \compl{2\ell}\}$, $\ell \neq n$.}

In this case, $T$ has the form
\[
{
\yhwratio{1:2}
T = \begin{young}
U & V & U^\circ
\end{young}
}
\]
where
\[
U = \begin{young}[c]
a_1 & a_2 & ]= & \dots & =] & ]=] a_\ell \\
b_1 & b_2 & ]= & \dots & =] & ]=] b_\ell \\
\end{young}
\qquad\text{and}\qquad
U^\circ = \begin{young}[c]
\compl{b_\ell} &  ]= & \dots & =] & ]=] \compl{b_2} & \compl{b_1} \\
\compl{a_\ell} &  ]= & \dots & =] & ]=] \compl{a_2} & \compl{a_1} \\
\end{young}
\]
are rotational complements of each other, and 
$V$ is rotationally symmetric.
The entries $a_1, \dots, a_\ell, b_1, \dots, b_\ell$ are the numbers 
$1, \dots, 2\ell$ in some order.  
$\web_T$ has arcs
$\{a_1, b_{\pi(1)}\}, \dots, \{a_\ell, b_{\pi(\ell)}\}$ coming from $U$,
$\{\compl{a_1}, \compl{b_{\pi(1)}}\}, \dots, \{\compl{a_\ell}, 
\compl{b_{\pi(\ell)}}\}$ coming from $U^\circ$, and $n-2\ell$ more arcs
coming from $V$.
Here $\pi$ is some permutation of $\{1, \dots, \ell\}$.

Let $T' = V-2\ell$. The arcs of $\fold(\web_T)$ are
$\{\compl{2a_i}+1, \compl{2b_{\pi(i)}}\}$ and
$\{\compl{2a_i}, \compl{2b_{\pi(i)}}+1\}$, $i=1, \dots \ell$, together
with the arcs of $\fold(\web_{T'})$.

\begin{lemma}
\label{lem:row-intact}
$P^{2\ell}(T) +2\ell = 
{\yhwratio{1:3}
\begin{young}
V  & U^{\circ} & U{+}N
\end{young}}\,.$  
Furthermore, for $1 \leq j \leq 2\ell$, $P^j(T) + j$ is of the form
\[
   \begin{young}[c]
]= & & & & & =]
&  \compl{b_\ell} &  &  \dots &  & =] \compl{b_1}
& & =]
\\
]= & & & =]
&  \compl{a_\ell} &  &  \dots &  & =] \compl{a_1}
& ]= & & & =]
\end{young}
\,,
\]
i.e., each row of $U^\circ$ is a subtableau of $P^j(T) + j$.
\end{lemma}

\begin{proof}
The first statement follows from Lemma \ref{promotion-rectification}. 
The second follows from the first, since none of the entries $\compl{a_i}$
can slide upward in computing any of the intermediate tableaux $P^j(T)+j$.
\end{proof}

From the first statement it follows that
\[
   \fold(T) = 
{\yhwratio{1:3}
\begin{young}
\fold(T')  & S
\end{young}}
\] for some $2 \times 2\ell$ tableau $S$.  
From the second statement, taking $j=a_i$, 
we see that $\compl{2a_i} = \compl{a_i}-j$ 
is in the second row of $P^j(T)$.  
Lemma~\ref{lem:promotion-folding} implies that $\compl{2a_i}$ is 
in the second row of $\fold(T)$.  Similarly, taking $j=b_i$, we find that
$\compl{2b_i}$ is the first
row of $\fold(T)$.  Since $\fold(T)$ is a domino tableau the only
possibility for $S$ is
\[
S = 
\text{\small$
\begin{young}[7ex][c]
\compl{2b_\ell} & \compl{2b_\ell}{+}1  & ]= & \dots & =] &
]=] \compl{2b_2} & \compl{2b_2}{+}1  
&\compl{2b_1} & \compl{2b_1}{+}1   \\
\compl{2a_\ell} & \compl{2a_\ell}{+}1  & ]= & \dots & =] &
]=] \compl{2a_2} & \compl{2a_2}{+}1  
&\compl{2a_1} & \compl{2a_1}{+}1   \\
\end{young}
$}
\,.
\]
Hence the arcs of $\web_{\fold(T)}$ are
$\{\compl{2a_i}+1, \compl{2b_{\pi(i)}}\}$ and
$\{\compl{2a_i}, \compl{2b_{\pi(i)}}+1\}$, $i=1, \dots \ell$, together
with the arcs of $\web_{\fold(T')}$.  Again, by the inductive hypothesis,
$\web_{\fold(T')} = \fold(\web_{T'})$, and we conclude that 
$\fold(\web_T)$ and  $\web_{\fold(T)}$ have the same arcs. \qed


\section{Proof of Theorem \ref{thm:fw1}} 

In this section, we prove Theorem~\ref{thm:fw1}.
Let $\web$ be a symmetrical web, and $T_\web$ be the corresponding tableau.
As in Algorithm~\ref{alg:fw1}, let $h_j = \webdist(B_j, B_{N-j})$,
and also define 
\[
g_j = \webdist(B_j, B_{N-j+1})
= \webdist(B_{j-1}, B_{N-j}).
\]
Here, the second equality is from the symmetry of $\web$.

Let $u = u_1 u_2 \dots u_N$ be the row-index word of $F(T_\web)$.  We need
to show that $u = w$, where $w$ is the word defined by 
Algorithm~\ref{alg:fw1}.

\begin{lemma}
For $1 \leq j \leq \fhalfN$,
we have 
\[
   u_{\compl{2j}} = \Phi(h_j-g_j) 
\qquad\text{and}\qquad
  u_{\compl{2j}+1} = \Phi(g_j-h_{j-1})
\,,
\]
where $\Phi$ is the function defined in Algorithm~\ref{alg:tw1}.
\end{lemma}

\begin{proof}
Let $\web' = P^j(\web)$; denote its internal boundary faces by 
$B_1', B_2', \dots, B_{N-1}'$, and write $B_0' = B_N'$ for the 
outer face.
Since $P^j(\web)$ is just $\web$ cyclically rotated $j$ times,
each boundary face $B_k'$ of $\web'$ is isomorphically identified with 
boundary face $B_{k'}$ of $\web$, where $k' = k+j \pmod N$.  
In particular, we have the following identifications:
\[
B_0'  \leftrightarrow B_j 
\qquad \qquad
B_{\compl{2j}-1}'  \leftrightarrow B_{N-j} 
\qquad \qquad
B_{\compl{2j}}'  \leftrightarrow B_{N-j+1} 
\]

By Lemma~\ref{lem:promotion-folding}, the position of entry $\compl{2j}$ in
$\fold(T_\web)$ is the same as its position in 
$P^j(T_\web) = T_{P^j(\web)}$.  The row which contains any given entry
of $T_{P^j(\web)}$ can be determined from $P^j(\web)$, using 
Algorithm~\ref{alg:tw1}.  In particular, the algorithm asserts that
$\compl{2j}$ is in row 
\[
\Phi\left(
\webdist(B_0',B_{\compl{2j}-1}') - \webdist(B_0',B_{\compl{2j}}')\right)
\,.
\]
From the identifications above, we see that this is precisely $\Phi(h_j - g_j)$.

A similar argument, using Lemma~\ref{lem:promotion-folding-symmetric},
shows that $\compl{2j}+1$ is in row $\Phi(g_j - h_{j-1})$ of $F(T_\web)$.
\end{proof}

Now, since entries $\compl{2j}$ and $\compl{2j}+1$ form a domino in 
$\fold(T_\web)$,
there are five possible values for the pair 
$(u_{\compl{2j}}, u_{\compl{2j}+1})$:
\[
   (1,1),\ (1,2),\ (2,2),\ (2,3),\ \text{or }(3,3)
\,.
\]
For each possibility, the following holds:
\begin{align*}
   \Lambda^{-1}(u_{\compl{2j}}, u_{\compl{2j}+1}) 
  &= 
   \Phi^{-1}(u_{\compl{2j}}) + \Phi^{-1}(u_{\compl{2j}+1}) \\
  &= 
   (h_j-g_j) + (g_j - h_{j-1}) \\
  &= 
   h_j- h_{j-1}
\,.
\end{align*}
Thus 
$(u_{\compl{2j}}, u_{\compl{2j}+1}) = \Lambda(h_j-h_{j-1}) = 
(w_{\compl{2j}}, w_{\compl{2j}+1})$, for $1 \leq j \leq \fhalfN$.  
In the case where $N$ is odd, we also have $u_1 = w_1 = 1$.  Therefore,
$u = w$, as required.
\qed


\section{Proof of Theorem~\ref{thm:fw2}}


In this section, we prove Theorem~\ref{thm:fw2}.  Since the odd and even cases are conceptually very similar, but with a number of minor differences in details, we will assume here that $n$ even.  The details of the odd case are left as an exercise to the reader.  Throughout, we assume that $\domtab$, $\rdomtab$, $V$, $C$, $\mdiagcross_\domtab$, $\webcross_\domtab$, $M$, $N$, etc. are as in the statement of Algorithm~\ref{alg:fw2-even}.  

In {\tt Step 5} of the algorithm, we asserted that if $ (k_1, k_2) \in V$, then $(k_1, k_2)$ forms an arc in the m-diagram of $C$. We begin by establishing this claim and other properties of these special arcs.

\begin{lemma} \label{lemma:max-arc}
 Suppose $ (k_1, k_2) \in V$.  Then $ (k_1, k_2) $ is an arc in $\mdiag_\compresstab$. No two such arcs intersect.  Moreover, if $ (k_1, k_2)$ is a first arc in $\mdiag_\compresstab$, then it is a maximal first arc, i.e. there are no first arcs above it; if $(k_1,k_2)$ is a second arc in $\mdiag_\compresstab$, then it is a maximal second arc, i.e. there are no second arcs above it.
 \end{lemma}

\begin{proof}
Without loss of generality, suppose $(k_1, k_2)$ originates from a simple block of type 1. Compressing the simple blocks of $\domtab$ using {\tt Step 3} of Algorithm~\ref{alg:fw2-even} yields rectangular tableaux whose concatenation form $\compresstab$. Moreover, $k$ labels a domino on the top (resp. bottom) two rows of $\rdomtab$ if and only if  $k$ is on the top (resp. bottom) two rows of $\compresstab$. By our assumption, the top two rows of $\rdomtab$ has the form 
\[ \vc{\includegraphics[scale=0.3]{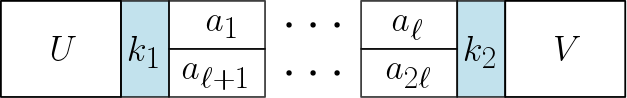}}\]  
where $u<k_1$ for all domino labels $u$ in $U$, $v>k_2$ for all domino labels $v$ in $V$, and $k_1 < a_1, \cdots , a_{2\ell} < k_2$ for all domino labels in between. 
It follows that the top two rows of $\compresstab$ has the form 
\[ \vc{\includegraphics[scale=0.3]{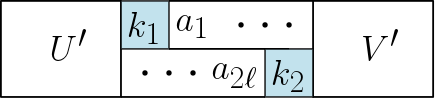}}\]  
where $u'<k_1$ for all $u'$ in $U'$ and $v' >k_2$ for all $v'$ in $V'$.  This implies that $(k_1, k_2)$ forms a maximal first arc in $\mdiag_C$.  

For distinct vertical domino pairs $ (k_1, k_2),\; (l_1, l_2) \in V$, either $\max{k_1,k_2} < \min{l_1, l_2}$ or $\max{l_1, l_2}<\min{k_1, k_2}$. Therefore, arcs in $\mdiag_C$ from vertical domino pairs do not intersect.  
\end{proof}

Next, we consider which pairs of arcs can intersect in $\mdiagcross_\domtab$.
We remind the reader that, by ``intersect'', we mean that the arcs cross 
or have a common endpoint.

Suppose $\alpha = (a,b)$ and $\beta = (c,d)$ are two arcs
in an arc diagram that intersect.  We will say that
$\alpha$ is \emph{directed toward} $\beta$ if $a = \min{a,b,c,d}$ or
$a = \max{a,b,c,d}$.  Similarly $\alpha$ is \emph{directed away from} $\beta$
if $b = \min{a,b,c,d}$ or $b= \max{a,b,c,d}$.

We classify arcs in 
$\mdiagcross_\domtab$
as either first arcs or second arcs according to which type of arc they
are derived from the diagram $\mdiag_C$.  If $(k_1, k_2)$ is a first
arc of $\mdiag_C$ then any primed variants of this arc 
in $\mdiagcross_\domtab$ --- i.e.
$(k_1, k_2)$, $(k_1',k_2')$, $(k_1', k_2)$ or $(k_1,k_2')$ ---
will also be called \emph{first arcs}.
Similarly, if $(k_1, k_2)$ is a second arc of $\mdiag_C$, then any of these
variants are \emph{second arcs} 
$\mdiagcross_\domtab$
If we have two arcs which are both first arcs or both second arcs, we will 
say that they have the same \emph{type}.
Arcs of the form $(k_1, k_2)$ or $(k_1', k_2')$ are called
\emph{uncrossed arcs}, and arcs of the form $(k_1', k_2)$ or $(k_1, k_2')$ 
are called \emph{crossed arcs}.

\begin{lemma}
\label{lem:crossing-arcs}
If two arcs in $\mdiagcross_\domtab$ intersect,
then one of the following must be true.
\begin{enumerate}
\item[(a)] 
Both arcs are uncrossed arcs, and they are opposite type.
\item[(b)]
Both arcs are crossed arcs, and they are reflections of each other.
\item[(c)]
One arc is a crossed first arc, the other is an 
uncrossed second arc, and the latter is directed toward the former.
\item[(d)]
One arc is a crossed second arc, the other is an uncrossed first arc,
and the latter is directed away from the former.
\end{enumerate}
\end{lemma}

\begin{proof}
Since uncrossed first and second arcs each form their $2$-webs, uncrossed arcs intersect each other only when one is a first arc and the other is a second arc. While crossed arcs intersect their reflections, it is impossible for the crossed arcs from two different vertical domino pairs $(k_1, k_2)$ and  $(l_1, l_2) $ to intersect, because either $\max{k_1, k_2} < \min{l_1, l_2} $ or $\max{l_1, l_2} < \min{k_1, k_2}$. 

Now suppose one arc is crossed and the other arc is not crossed.
Since crossed arcs originate from maximal first or second arcs in $\mdiag_C$, 
it is impossible for both arcs to be of the same type.
Suppose a crossed first arc $(a', b)$ or $(a,b')$  intersects an uncrossed second arc of the form $(k_1, k_2) $ or $(k_1',k_2')$ where $k_1>k_2$. Then either $k_2 \le b \le k_1$ or $k_2 \le a \le k_1$, implying that the uncrossed second arc is directed toward the crossed first arc. Similarly, when a crossed second arc intersects an uncrossed first arc, the latter is directed away from the former.
\end{proof}

 \begin{lemma}
  $\webcross_\domtab$ is a $3$-web.  
 \end{lemma} 

 \begin{proof}
   The resolution of the crossed m-diagram in {\tt Step 6} gives a bipartite directed graph in the plane satisfying conditions (1)--(3) in the definition of a $3$-web (Definition~\ref{def:3-web}).  It remains to check that every internal face of $\webcross_\domtab$ has at least $6$ sides.

   All the internal faces of $ \webcross_\domtab$ have an even number of sides 
because the graph is bipartite.
Internal faces in $ \webcross_\domtab$ originate from regions bounded by intersecting arcs in the crossed m-diagram. Since two arcs can intersect at most once, the only possible way to form an internal face with less than $6$ sides is   to resolve a region bounded by $3$ or $4$ intersecting arcs. Enumerating all possibilities shows that the following two cases resolve into faces with $4$ sides. 
\begin{center}
\begin{tabular}{cc}
~~~~~~~~~~~~\includegraphics[scale=0.3]{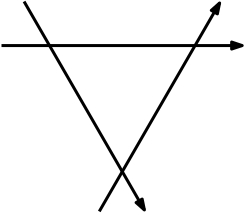}~~~~~~~~~~~~  &
~~~~~~~~~~~~\includegraphics[scale=0.3]{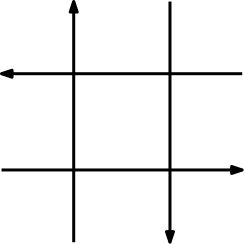}~~~~~~~~~~~~ \\
 Case 1   &    Case 2   \\
\end{tabular}
\end{center}

Suppose, for a contradiction,  that Case 1 occurs in a crossed m-diagram.  A simple analysis of the cases in Lemma~\ref{lem:crossing-arcs} shows that if three arcs intersect each other pairwise, then two of the arcs must be crossed arcs which are reflections of each other, and the other must be an uncrossed arc of the opposite type.
%
There are two configurations where an uncrossed arc intersects two such crossed arcs: 
\begin{center}
	\begin{tabular}{cc}
	  ~~~~ \includegraphics[scale=0.27]{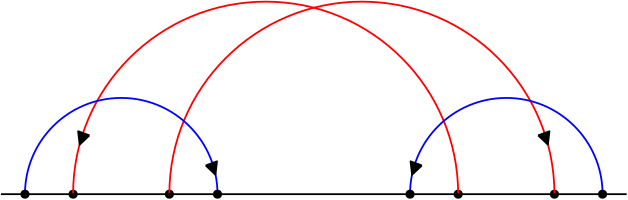}   ~~~~  &  ~~~~	 \includegraphics[scale=0.27]{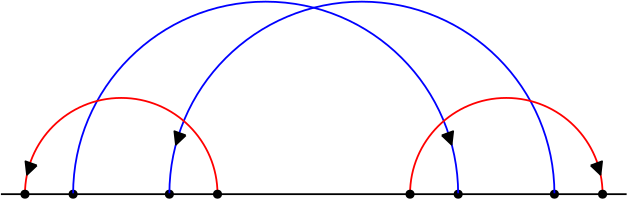}   ~~~~ \\ when the crossed arcs are first arcs   &  when the crossed arcs are second arcs   
	\end{tabular}	
\end{center}
Neither of the two configurations matches Case 1.  Therefore, Case 1 cannot occur in a crossed m-diagram. 

Suppose, for a contradiction, that Case 2 occurs in a crossed m-diagram.  Since we already know that ordinary m-diagrams cannot resolve to include a face with $4$ sides, one of the arcs must be a crossed arc.  We label the arcs as follows.  
 \begin{center}
\includegraphics[scale=0.3]{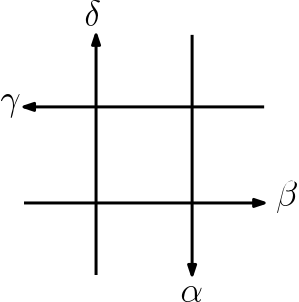} 
\end{center}
Without loss of generality, suppose $\alpha $ is a crossed arc. If $ \beta$  is a crossed arc, then $\beta$ is $\alpha$'s reflection. Then both $\gamma$ and $\delta$ are uncrossed arcs   of the same type. But this is impossible,  as $ \gamma$ and $\delta$ intersect.  It follows that $\beta$ is an uncrossed arc, and so is $\gamma$ for the same reason.  However, one directs toward $\alpha$ and the other directs away from $\alpha$, contradicting   Lemma~\ref{lem:crossing-arcs}. Therefore, Case 2 cannot occur in a crossed m-diagram. 
 \end{proof}


Suppose that $\web$ is obtained by resolving a 
generalized m-diagram $\mdiag$.  Although $\mdiag$ itself is not
necessarily planar,
there is an associated planar graph, obtained by replacing each crossing
with a vertex.  We will refer to the faces of this associated graph as 
faces of $\mdiag$.  The resolution process gives a canonical one-to-one 
correspondence between the faces of $\mdiag$ and the faces of $\web$.
We will use this correspondence implicitly in our notation: 
if $X$ is a face of
$\web$, we will also refer to the corresponding face of $\mdiag$ as $X$,
and vice-versa.  

For any face $X$ of $\mdiag$, define $\arcset(X)$ to be the set of arcs
that are above $X$.  If $X$ and $Y$ are two faces, define
\[
   \arcset(X,Y) = \arcset(X) \symmdiff \arcset(Y)
\,,
\]
the symmetric difference of the two sets of arcs.  Thus $\arcset(X,Y)$ is 
the set of arcs that separate $X$ and $Y$.  We define the \emph{arc-distance}
between $X$ and $Y$ to be
\[
   \arcdist(X,Y) = |\arcset(X,Y)|
\,.
\]
This is also the distance between faces $X$ and $Y$ in the planar graph 
associated to $\mdiag$.

Note that adjacency of faces is different for $\mdiag$ and 
$\web$: if $X$ and $Y$ are adjacent faces in $\mdiag$,
then $X$ and $Y$ are also adjacent in $\web$, but the converse is not true.
Thus, in general, $\arcdist(X,Y) \geq \webdist(X,Y)$.

Suppose $\alpha = (a,b)$ and $\beta = (c,d)$ are two arcs of $\mdiag$
that intersect.
Then $\alpha$ and $\beta$ divide the half-plane above the boundary line into 
four regions (three if the intersection point is on the boundary line).
Denote these regions $R_{ac}$, $R_{ad}$, $R_{bc}$ and $R_{bd}$,
where $R_{kl}$ is the region incident with vertices $k$ and $l$.
Note that if the 
intersection
of $\alpha$ and $\beta$ is resolved, then
regions $R_{ad}$ and $R_{bc}$ become adjacent, whereas 
$R_{ac}$ and $R_{bd}$ remain non-adjacent.
Every face of $\mdiag$ is a subset of one of these four regions.  
Suppose $X$ and $Y$ are two faces of $\mdiag$.  We
will say that $\{\alpha, \beta\}$ \emph{coherently separates} $X$ and $Y$ if
$X \subset R_{ad}$ and $Y \subset R_{bc}$, or $X \subset R_{bc}$ and
$Y \subset R_{ac}$.  We write $\cosep(X,Y)$ for the set of all pairs 
of arcs that coherently separate $X$ and $Y$.

\begin{lemma} 
\label{lem:webdist-bounds}
Let $X, Y$ be faces in a web $\web$, which is the resolution
of a generalized m-diagram $\mdiag$.  Then 
\[
   \webdist(X,Y) \geq \arcdist(X,Y) - |\cosep(X,Y)|
\,.
\]
\end{lemma}

\begin{proof}
%
Suppose $X = X_0, X_1, X_2, \dots, X_\ell = Y$ is a shortest path 
from $X$ to $Y$ in $\web$.  Let $S = \arcset(X,Y)$.  For $i=1 \dots, \ell$, put
\[
  S_i = \big(S \cap \arcset(X_{i-1},X_i)\big ) \setminus \big (\textstyle \bigcup_{j=1}^{i-1} S_j \big)
\,.
\]
Informally, $S_i$ is the set of arcs separating $X$ and $Y$, which are 
crossed at the $i$-th step in the path, and have not been crossed previously.
Note that since $S = \arcset(X_0, X_1) \symmdiff \dotsb \symmdiff \arcset(X_{\ell-1}, X_\ell)$, we have
$S \subseteq \bigcup_{i=1}^\ell \arcset(X_{i-1}, X_i)$.
By construction the sets $S_1, \dots, S_\ell$ are disjoint and contained
in $S$.  Therefore $S = S_1 \sqcup S_2 \sqcup \dots \sqcup S_\ell$.

Since $X_{i-1}$ and $X_i$ are adjacent faces in $\web$,
$|S_i| \leq |\arcset(X_{i-1},X_i)| \leq 2$.  Moreover, 
we claim that if $|S_i| = 2$, then $S_i \in \cosep(X,Y)$.
To see this suppose $S_i = \{\alpha, \beta\}$, 
$\alpha = (a,b)$, $\beta = (c,d)$.  Since $X_i$ and $X_{i-1}$
are adjacent in $\web$ then the 
intersection
of $\alpha$, $\beta$ resolves
so that the intersection edge is incident with $X_i$ and $X_{i-1}$. Thus,
$S_i \in \cosep(X_{i-1},X_i)$.
Without loss of generality, assume $X_{i-1} \subset R_{ad}$ and 
$X_i \subset R_{bc}$.  By definition of $S_i$,
$\alpha$ and $\beta$ are not in $\arcset(X_{j-1}, X_j)$ for any $j <i$.
This means
$X = X_0, X_1, \dots, X_{i-1}$ are all contained in $R_{ad}$.
Since $\alpha$ and $\beta$ are both in $\arcset(X,Y)$, $Y$ is contained
in the opposite region: $Y \subset R_{bc}$, which proves the claim.

Since $|S| = \arcdist(X,Y)$ and  $\#\{i \mid |S_i| = 2\} \leq |\cosep(X,Y)|$, 
we have
\begin{align*}
\webdist(X,Y) = \ell &= |S| + \#\{i \mid |S_i| = 0\} - \#\{i \mid |S_i| = 2\} \\
&\geq \arcdist(X,Y) + 0 - |\cosep(X,Y)|
\,. 
\qedhere
\end{align*}
\end{proof}

We now apply this lemma to the crossed m-diagram $\mdiagcross_\domtab$.  For a face
$X$ of $\mdiagcross_\domtab$, we will say that $X$ is \emph{between a vertical pair}
if there exists a vertical pair $(k_1, k_2) \in V$ 
such that exactly one of $(k_1', k_2)$ and $(k_1,k_2')$ is in $\arcset(X)$.
Let 
\[
   \epsilon(X) = \begin{cases}
 1&\quad\text{if $X$ is between a vertical pair}  \\
 0 &\quad\text{otherwise.}
\end{cases}
\]

\begin{lemma}
\label{lem:symmetrical-distance-1}
Assume that $\mdiagcross_\domtab$ is symmetrical. 
Let $X$ be a face of $\mdiagcross_\domtab$ and let $X'$ be its reflection.  
Then 
\[ 
   \webdist(X,X') = \arcdist(X,X') - \epsilon(X)
\,.
\]
\end{lemma}

\begin{proof}
Suppose $X$ is not between
a vertical pair. By considering each of the cases in 
Lemma~\ref{lem:crossing-arcs} we see that $\cosep(X,X') = \emptyset$.
By Lemma~\ref{lem:webdist-bounds}, $\webdist(X,X') = \arcdist(X,X')$.

Now suppose $X$ is between
the vertical pair $(k_1,k_2) \in V$. 
In this case, Lemma~\ref{lem:crossing-arcs} 
implies that $\cosep(X,X') = \big\{\{(k_1,k_2'), (k_1',k_2)\}\big\}$.
Without loss of generality, assume that $X \in R_{k_1k_2}$, and
$X' \in R_{k_1'k_2'}$.  Let $Y$ be the unique face in $R_{k_1k_2}$
incident with the intersection point of the arcs $(k_1,k_2')$ 
and $(k_1',k_2)$, and let $Y'$ be its reflection.  We claim that
\[
 \arcset(X,X') = \arcset(X,Y) \sqcup \arcset(X',Y') \sqcup \arcset(Y,Y')
\,.
\]
First, note that $\arcset(Y,Y') = \{(k_1,k_2'), (k_1',k_2)\}$.
By the construction of $Y$, each arc in $\arcset(Y,Y')$ is either in both 
$\arcset(X)$ and $\arcset(Y)$, or neither.  Thus $\arcset(Y,Y')$
is disjoint from $\arcset(X,Y)$, and similarly, $\arcset(Y,Y')$ is disjoint 
from $\arcset(X',Y')$.
Now suppose there is an arc $\gamma \in \arcset(X,Y) \cap \arcset(X', Y')$.
Since $\gamma$ separates $X$ and $Y$, it must pass through $R_{k_1k_2}$, 
and since $\gamma$ separates $X'$ and $Y'$ it must pass through 
$R_{k_1'k_2'}$.  Thus $\gamma$ must be a crossed arc, which intersects both 
$(k_1,k_2')$ and $(k_1',k_2)$.  But this is impossible by 
Lemma~\ref{lem:crossing-arcs},
so $\arcset(X,Y)$ and $\arcset(X',Y')$ are disjoint as well.  The
claim now follows, by noting that
\[
 \arcset(X,X') = \arcset(X,Y) \symmdiff \arcset(X',Y') \symmdiff \arcset(Y,Y')
\,.
\]
We conclude that $\arcdist(X,X') = 2\,\arcdist(X,Y) + 2$, and
\begin{align*}
  \webdist(X,X') 
  &\leq \webdist(X,Y) + \webdist(Y,Y') + \webdist(Y',X') \\
  &\leq 2\,\arcdist(X,Y) +1 \\
  &=  \arcdist(X,X') -1 \,.
\end{align*}
Since $|\cosep(X,X')|= 1$, the reverse inequality follows from 
Lemma~\ref{lem:webdist-bounds}. 
\end{proof}

We now focus on distances between boundary faces, as these are the
quantities which appear in 
Algorithms~\ref{alg:tw1} and \ref{alg:fw1}.

Let $B_0, B_1, \dots, B_M$ denote 
the boundary faces of the web $\web_C$,
as described in Section~\ref{sec:webs}.
Here, $B_0 = B_M$ is the outer face, and $M = \frac{N}{2}$.

Since the vertices of $\mdiagcross_\domtab$ are labelled 
$M', \dots ,1' ,1, \dots, M$ instead of
$1, \dots, N$, we will adjust our notation for its boundary faces, as follows.
Let $A_k$ be the boundary face of $\mdiagcross_\domtab$ incident with vertices
$k$ and $k+1$; let $A_{k'}$ be the boundary face incident with
$k'$ and $(k+1)'$; let $A_0 = A_{0'}$ be the boundary face incident
with $1$ and $1'$; $A_M = A_{M'}$ is the outer
face.

\begin{lemma}
\label{lem:symmetrical-distance-2}
For $k \in \{0, \dots, M\}$, we have
\[
   \arcdist(A_k, A_{k'}) = 2\,\webdist(B_k, B_0)
\,.
\]
\end{lemma}

\begin{proof}
From \cite[Lemma 4.5]{tymoczko}, 
$ \webdist ( B_k , B_0) = \arcdist (B_k, B_0)$. 
Moreover, since $B_0$ is the outer face of $\mdiag_C$, 
$\arcdist (B_k, B_0) = |\arcset(B_k)|$. 

If $A_k$ is not between a vertical pair, 
then $\arcset(A_k) \setminus \arcset(A_{k'}) = \arcset(B_k)$.
If $A_k$ is between vertical pair $(k_1, k_2)$,
then $(k_1, k_2) \in \arcset(B_k)$ and
\[
\arcset(A_k) \setminus \arcset(A_{k'}) = 
\begin{cases} 
\arcset(B_k) \cup \{(k_1', k_2)\} \setminus \{(k_1, k_2)\} 
&\text{if }k_1 < k_2 \\
\arcset(B_k) \cup \{(k_1, k_2')\} \setminus \{(k_1, k_2)\} 
&\text{if }k_1 > k_2. \\
\end{cases}
\]
In either case 
$|\arcset(A_k) \setminus \arcset(A_{k'})| = |\arcset(B_k)|$.
By a symmetrical argument, $|\arcset(A_{k'}) \setminus \arcset(A_k)| 
= |\arcset(B_{k'})|$.  Thus
\[
   |\arcset(A_k, A_{k'}) | = 
|\arcset(A_k) \setminus \arcset(A_{k'})|
+
|\arcset(A_{k'}) \setminus \arcset(A_k)| 
= 2 |\arcset(B_k)|
\]
as required.
\end{proof}

Finally we are ready to complete the proof of Theorem~\ref{thm:fw2}.

Let $\web = \webcross_\domtab$.  Since $\web$ is a $3$-web,
by Proposition~\ref{prop:independent-of-drawing}, we may assume 
that $\mdiagcross_\domtab$ is drawn symmetrically.
Hence $\web$ may also be assumed to be 
symmetrical.  Therefore $\web$ has an associated domino tableau
$D_\web$, defined by Algorithm~\ref{alg:fw1}.  By Theorem~\ref{thm:fw1} it suffices to show that $D = D_{\web}$.

To this end, we compare the row-index words.
Let $u = u_1 u_2 \dots u_N = w(D)$, and
let $w = w_1 w_2 \dots w_N = w(D_{\web})$.
In addition, let $v = v_1 v_2 \dots v_M = w(C)$.
Our goal is to show that $u=w$.  We will do so by showing that $u$ and
$w$ are related to $v$ in the same way.

For $j \in \{0, \dots, M\}$, 
let $d_j = \webdist(B_j, B_0)$ be the distances in $\web_C$,
defined in Algorithm~\ref{alg:tw1}.
Let $h_j$ be distances in $\web$,
defined in Algorithm~\ref{alg:fw1}.  After adjusting for 
the different vertex-labelling in $\mdiagcross_\domtab$, 
the definition is $h_j = \webdist(A_{M-j}, A_{(M-j)'})$.

Now fix $j$, and let $k = M+1-j$.  Then by 
Lemmas~\ref{lem:symmetrical-distance-1} 
and~\ref{lem:symmetrical-distance-2}, 
\begin{align*}
    h_j 
    &= \webdist(A_{M-j}, A_{(M-j)'}) \\
    &= \arcdist(A_{M-j}, A_{(M-j)'}) - \epsilon(A_{(M-j)})\\
    &= 2\,\webdist(B_{M-j}, B_0) - \epsilon(A_{(M-j)})\\
   &= 2d_{k-1} - \epsilon(A_{k-1}).
\end{align*}
Let $\delta_k = \epsilon(A_k) - \epsilon(A_{k-1})$.  
From 
Algorithm~\ref{alg:tw1}, $v_k = \Phi(d_{k-1} - d_k)$. From the definition
of $w$ in
Algorithm~\ref{alg:fw1} and the preceding calculations, we have
\begin{align*}
   (w_{\compl{2j}}, w_{\compl{2j}+1}) 
   &= \Lambda(h_j - h_{j-1})  \\
   &= \Lambda(2(d_{k-1} - d_k) - \epsilon(A_{k-1}) + \epsilon(A_k)) \\ 
   &= \Lambda(2 \Phi^{-1}(v_k) + \delta_k)
\,.
\end{align*}

Observe that $\delta_k = 0$,
unless one of $A_k$ and $A_{k-1}$ is between a vertical pair and the other 
is not --- this occurs if and only if $k$ part of a vertical pair.
More precisely,
\[
   \delta_k =
\begin{cases}
 0 & \quad \text{if $k$ is a horizontal domino in $\rdomtab$}  \\
 1 & \quad \text{if $k$ is a first vertical domino in $\rdomtab$} \\
 -1 & \quad \text{if $k$ is a second vertical domino in $\rdomtab$.} \\
\end{cases}
\]
The domino labelled by $k$ in $\rdomtab$ corresponds to
the domino labelled by the pair $(\compl{2j}, \compl{2j}+1)$ in $D$.
Thus, $(u_{\compl{2j}}, u_{\compl{2j}+1})$ tells us which rows
contain this domino.  We consider the various cases.

\begin{itemize}
\item  If $(u_{\compl{2j}}, u_{\compl{2j}+1}) = (r,r)$, then $k$ is
a horizontal domino in row $r$ of $\rdomtab$.  Therefore $k$ is also 
in row $r$ of $C$.  Hence, in this case, $v_k = r$ and  $\delta_k = 0$.
\item If $(u_{\compl{2j}}, u_{\compl{2j}+1}) = (r,r+1)$, then $k$ is a
vertical domino in rows $r$ and $r+1$ of $\rdomtab$.  
\begin{itemize}
\item If $k$ is a first vertical
domino, then $k$ is in row $r$ of $C$; hence $v_k=r$ and $\delta_k = 1$.
\item
If $k$ is a second vertical
domino, then $k$ is in row $r+1$ of $C$; hence $v_k=r+1$ and $\delta_k = -1$.
\end{itemize}
\end{itemize}
In each case, we can verify directly that we have the relation
\[
   (u_{\compl{2j}}, u_{\compl{2j}+1}) 
   = \Lambda(2 \Phi^{-1}(v_k) + \delta_k)
\,.
\]
Thus
$(u_{\compl{2j}}, u_{\compl{2j}+1}) =  (w_{\compl{2j}}, w_{\compl{2j}+1})$, 
as required.
\qed

\section{Concluding remarks}

In the introduction, we claimed that if $\web$ is a symmetrical $3$-web, 
then $\fold(\web)$, the web of $F(T_\web)$,
does not look like a ``doubled version'' of the right half $\web$.  
This is a somewhat subjective
claim, but there is content to it.  One can try to establish a
relationship by drawing $\fold(\web)$ with its arcs and vertices within
some small distance $\varepsilon$ of the arcs and vertices on the
right half of $\web$.  
It is always possible to do this, but when one does, no obvious pattern 
emerges.  
Even in small examples one can see that some of the arcs 
of $\web$ have several arcs of $\fold(\web)$
nearby, while others have only one.  It is certainly nothing like a
two-to-one mapping.  However, it is still conceivable that there 
is a less obvious pattern.

Webs originated in representation theory, as a way to produce bases for
invariant subspaces of representations of 
$U_q(\mathfrak{sl}_2)$ and $U_q(\mathfrak{sl}_3)$.
This raises the question: do any of the results in this paper
have a representation theoretic interpretation?

On a more combinatorial level, the concept of a generalized m-diagram
raises some interesting questions.
Not every generalized m-diagram resolves to a $3$-web, 
satisfying conditions (1)--(4) of Definition~\ref{def:3-web}.
One might therefore hope to characterize generalized m-diagrams 
for which the resolution is a valid web.
A more refined problem: given a $3$-web $\web$, classify 
and/or enumerate generalized m-diagrams that resolve to $\web$.
There is always at least one, since the m-diagram of $T$
resolves to $\web_T$, but Theorem~\ref{thm:fw2} shows that there
may be others.
For $2$-webs, we have a nice characterization of the webs of domino
tableaux.  Is there a similar characterization of the $3$-webs of 
domino tableaux, e.g. in terms of some family of generalized m-diagrams?

There are various directions in which one might hope to generalize
this work.  Most ambitiously, one might hope for generalizations of
our main theorems
for an arbitrary web $\web$, not just for symmetrical webs.
More realistically, it may be possible to find generalizations to 
non-rectangular shapes, which may potentially involve non-rectangular
self-evacuating tableaux, rotationally symmetric skew tableaux,
and webs in which non every boundary vertex is a sink.

\begin{ack}
We thank Oliver Pechenik for helpful feedback on the first version
of this manuscript.
\end{ack}

\medskip

\bibliographystyle{plain}
\bibliography{references}


\end{document}